\documentclass[11pt,reqno,letterpaper]{amsart}
\usepackage{color}
\usepackage[colorlinks=true, allcolors=blue,backref=page]{hyperref}
\usepackage{amsmath, amssymb, amsthm}
\usepackage{mathrsfs}
\usepackage{mathtools}
\usepackage[noabbrev,capitalize,nameinlink]{cleveref}
\crefname{equation}{}{}
\usepackage{fullpage}
\usepackage[noadjust]{cite}
\usepackage{graphics}
\usepackage{pifont}
\usepackage{tikz}
\usepackage{bbm}
\usepackage[T1]{fontenc}

\usetikzlibrary{arrows.meta}

\usepackage{environ}
\usepackage{framed}
\usepackage{url}
\usepackage[linesnumbered,ruled,vlined]{algorithm2e}
\usepackage[noend]{algpseudocode}
\usepackage[labelfont=bf]{caption}
\usepackage{cite}
\usepackage{framed}
\usepackage[framemethod=tikz]{mdframed}
\usepackage{appendix}
\usepackage{graphicx}
\usepackage[textsize=tiny]{todonotes}
\usepackage{tcolorbox}
\usepackage{enumerate}
\allowdisplaybreaks[1]
\usepackage{enumerate}
\usepackage{stmaryrd}
\usepackage[margin=1in]{geometry}

\usepackage[shortlabels]{enumitem}
\crefformat{enumi}{#2#1#3}
\crefrangeformat{enumi}{#3#1#4 to~#5#2#6}
\crefmultiformat{enumi}{#2#1#3}
{ and~#2#1#3}{, #2#1#3}{ and~#2#1#3}

\DeclareSymbolFont{symbolsC}{U}{pxsyc}{m}{n}
\SetSymbolFont{symbolsC}{bold}{U}{pxsyc}{bx}{n}
\DeclareFontSubstitution{U}{pxsyc}{m}{n}
\DeclareMathSymbol{\medcircle}{\mathbin}{symbolsC}{7}

\crefname{algocf}{Algorithm}{Algorithms}

\crefname{equation}{}{} 
\AtBeginEnvironment{appendices}{\crefalias{section}{appendix}} 

\usepackage[color,final]{showkeys} 

\colorlet{refkey}{orange!20}
\colorlet{labelkey}{blue!30}

\crefname{algocf}{Algorithm}{Algorithms}

\numberwithin{equation}{section}
\newtheorem{theorem}{Theorem}[section]
\newtheorem{proposition}[theorem]{Proposition}
\newtheorem{lemma}[theorem]{Lemma}

\crefname{claim}{Claim}{Claims}

\newtheorem*{question*}{Question}

\theoremstyle{definition}
\newtheorem{definition}[theorem]{Definition}

\newtheorem*{definition*}{Definition}

\theoremstyle{remark}

\newcommand{\eps}{\varepsilon}

\renewcommand{\Pr}{\mathbb{P}}

\newcommand{\mbf}{\boldsymbol}

\renewcommand{\P}{\mathbb{P}}

\newcommand{\bs}{\boldsymbol}
\newcommand{\mb}{\mathbb}

\newcommand{\mr}{\mathrm}

\newcommand{\on}{\operatorname}

\crefname{algocf}{Algorithm}{Algorithms}

\crefname{equation}{}{} 
\AtBeginEnvironment{appendices}{\crefalias{section}{appendix}} 

\title{Sparse recovery properties of discrete random matrices}
\author[A1]{Asaf Ferber}
\address{Department of Mathematics, University of California, Irvine.}
\email{asaff@uci.edu}

\author[A2]{Ashwin Sah}
\author[A3]{Mehtaab Sawhney}
\address{Department of Mathematics, Massachusetts Institute of Technology, Cambridge, MA 02139, USA}
\email{\{asah,msawhney\}@mit.edu}

\author[A4]{Yizhe Zhu}
\address{Department of Mathematics, University of California, Irvine.}
\email{yizhe.zhu@uci.edu}

\thanks{Ferber was supported by NSF grants DMS-1954395 and DMS-1953799, NSF Career DMS-2146406, and Sloan's fellowship. Sah and Sawhney were supported by NSF Graduate Research Fellowship Program DGE-1745302. Sah was supported by the PD Soros Fellowship. Zhu was  supported by NSF-Simons Research Collaborations on the Mathematical and Scientific
Foundations of Deep Learning.}

\DeclareMathOperator{\supp}{supp}

\DeclareMathOperator{\Bad}{\boldsymbol{B}}

\begin{document}

\maketitle

\begin{abstract}
Motivated by problems from compressed sensing, we determine the threshold behavior of a random $n\times d$ $\pm 1$ matrix $M_{n,d}$ with respect to the property ``every $s$ columns are linearly independent''. In particular, we show that for every $0<\delta <1$ and $s=(1-\delta)n$, if $d\leq n^{1+1/2(1-\delta)-o(1)}$ then with high probability every $s$ columns of $M_{n,d}$ are linearly independent, and if $d\geq n^{1+1/2(1-\delta)+o(1)}$ then with high probability there are some $s$ linearly dependent columns.  
\end{abstract}

\section{Introduction}
 
Compressed sensing is a modern technique of data acquisition,  which is at the intersection of mathematics, electrical engineering, computer science, and physics, and has grown 
 tremendously in recent years. Mathematically, we define an unknown signal as a vector $\mbf{x}\in \mathbb R^d$, and we have access to \emph{linear measurements}: that is, for any vector $\mbf{a}\in \mathbb{R}^d$, we have access to $\mbf{a}\cdot \mbf{x}=\sum_{i=1}^da_ix_i$. In particular, if $\mbf{a}^{(1)},\ldots\mbf{a}^{(n)}\in \mathbb{R}^d$ are the measurements we make, then we have an access to the vector $\mbf{b}:=A\mbf{x}$, where 
\[
A:=\begin{pmatrix} -& \mbf{a}^{(1)} &-\\
& \vdots &\\
- &\mbf{a}^{(n)} &-\\
\end{pmatrix}.
\]
The tasks of compressed sensing are: $(i)$ to recover $\mbf{x}$ from $A$ and $\mbf{b}$ as accurately as possible, and $(ii)$ doing so in an efficient way. In practice, one would like to recover a high dimensional signal (that is, $d$ is large) from as few measurements as possible (that is, $n$ is small). In this regime, for an arbitrary vector $x\in \mathbb R^d$ the problem is ill-posed: for any given $\mbf{b}$, the solution of $\mbf{b}=A\mbf{x}$, if it exists, forms a (translation of) linear subspace of dimension at least $d-n$, and therefore there is no way to uniquely recover the original $\mbf{x}$.

A key quantity to look at to guarantee the success of (unique) recovery is the \emph{sparsity} of the vector $\mbf{x}$, and we say that a vector is $s$-\emph{sparse} if its \emph{support} is of size at most $s$. That is, if
\[|\mathrm{supp}(\mbf{x})|=\{i:x_i\not=0 \}\leq s.\]
A neat observation is that having at most one $s$-sparse solution to $A\mbf{x}=\mbf{b}$ for every $\mbf{b}$ is equivalent to saying that $A$ is $2s$-\emph{robust} (that is, every $2s$ columns of $A$ are linearly independent). Indeed, if we have two $s$-sparse vectors $\mbf{x}\neq \mbf{y}$ such that $A\mbf{x}=A\mbf{y}$ then $\mbf{x}-\mbf{y}$ is a nonzero $2s$-sparse vector in the kernel of $A$. For the other direction, if there is a nonzero $2s$-sparse vector in the kernel of $A$, one can split its support into two disjoint sets of size at most $s$ each and consider the vectors restricted to these sets, one of which is multiplied by $-1$. 

If we take $A$ to be a random Gaussian matrix $A$ (or any other matrix drawn from some ``nice'' continuous distribution), then we clearly have that with probability one $A$ is $s$-robust for $n=s$ and any $d\in \mathbb{N}$ (and in particular, one can uniquely recover $s/2$-sparse vectors). Moreover, in their seminal work, Candes and Tao \cite{CT06} showed that it is possible to efficiently reconstruct $\mbf{x}$ with very high accuracy by solving a simple linear program if we take $n=O(s\log(d/s))$.

In this paper, we are interested in the compressed sensing problem with integer-valued measurement matrices and with entries of magnitude at most $k$. Integer-valued measurement matrices have found applications in measuring gene regulatory expressions, wireless
communications,  and natural images \cite{zhang2019analysis,RHE14,he2010simplest}, and they are quick to generate and easy to store in practice \cite{Iwe14,XLJZ15}.
Under this setting, for integer-valued signal $\mbf{x}$, we can have exact recovery even if we allow some noise $\mbf{e}$ with $\|\mbf{e}\|_{\infty}<1/2$ (for more details, see \cite{FNS19}).

The first step is to understand when the compressed sensing problem is well-posed for given $s,n, k$, and $d$. Namely, for which values of $s,n,k$ and $d$ does an $s$-robust $n\times d$ integer-valued matrix with entries in $\{-k,\dots, k\}$ exist? For $s=n$, observe that if $d\geq(2k+1)^2n$, then by the pigeonhole principle, one can find $n$ columns for which their first two rows are proportional and therefore are not linearly independent. In particular, we have $d=O_k(n)$. In \cite{FNS19}, Fukshansky, Needell, and Sudakov showed that there exists an $s$-robust $A$ with $d=\Omega(\sqrt{k} n)$,  using the result of Bourgain, Vu and Wood \cite{BTV10} on the singularity of discrete random  matrices (in fact, the more recent result by Tikhomirov \cite{Tik18} gives a better bound for $k=1$). Konyagin and Sudakov \cite{KS20} improved the upper bound to $d=O(k\sqrt{\log k}n)$, and they gave a deterministic construction of $A$ when $d\geq \frac{1}{2} k^{n/(n-1)}>n$.

When $1\leq s\leq n-1$ and $k=2$, Fukshansky and Hsu \cite{fukshansky2019covering} gave a deterministic construction such that $d\geq \left(\frac{n+2}{2}\right)^{1+\frac{2}{3s-2}}$. When $s=o(\log n)$, this implies we can take $d=\omega(n)$. This result hints that if we allow $s$ to be ``separated away'' from $n$, then one could take $d$ to be ``very large''. A natural and nontrivial step to understanding the $s$-robustness property of matrices is to investigate the \emph{typical} behavior. For convenience, we will focus on the case $k=1$ (even though our argument can be generalized to all fixed $k$), and we define, for all $n,d\in \mathbb{N}$, the random variable $M_{n,d}$ which corresponds to an $n\times d$ matrix with independent entries chosen uniformly from $\{\pm 1\}$. For $1\leq s\leq n$, we would like to investigate the \emph{threshold} behavior of $M:=M_{n,d}$ with respect to being $s$-robust. That is, we wish to find some $d^*:=d(s,n)$ such that 
\[
\lim_{n\to\infty}\Pr[M \textrm{ is }s\textrm{-robust}]= \begin{cases} 0 & d/d^*\rightarrow \infty \\
1 & d/d^*\rightarrow 0.
\end{cases}
\]

It is trivial to show (deterministically) that if $s=n$ and $M$ is $s$-robust, then $d\leq 2n$. What if we allow $s$ to be ``separated away'' from $n$? That is, what if $s=(1-\delta)n$ for some $0<\delta<1$? It is not hard to show (and it follows from the proof of \cref{lem:large-rank}) that the probability for a random $n\times n$ matrix to have rank at least $(1-\delta)n$ is at least $1-2^{-\Omega(\delta^2n^2)}$. Therefore, one could think that a typical $M_{n,d}$ might be $(1-\delta)n$-robust for some $d=2^{n^{1-o(1)}}$. This turns out to be wrong as we show in the following simple theorem: 
\begin{theorem} \label{thm: det upper bound}
For any fixed $0<\delta<1$ there exists $C>0$ such that for sufficiently large $n\in \mathbb{N}$ the following holds. If $s= (1-\delta)n$ and $d\ge Cn^{1+1/(1-\delta)}$, then every $\pm 1$ $n\times d$ matrix $M$ is not $s$-robust.
\end{theorem}
\begin{proof}
Given any $s/2$-subset of column vectors $\mbf{v}_1,\ldots,\mbf{v}_{s/2}\in \{\pm1\}^n$ of $M$, by Spencer's ``six standard deviations suffice'' \cite{Spe85}, there exist some $x_1,\ldots,x_{s/2} \in \{\pm 1\}$ for which $\|\sum_{i=1}^{s/2}x_i\mbf{v}_i\|_{\infty}\leq C'\sqrt{n}$ for a universal constant $C'>0$ (a simple Chernoff bound suffices if one is willing to lose a $\sqrt{\log n}$ factor). Fix such a combination $\sum_{i=1}^{s/2}x_i\mbf{v}_i$ for each $s/2$-subset of column vectors. Since there are at most $\left(3C'\sqrt{n}\right)^{n}$ integer-valued vectors in the box $[-C'\sqrt{n},C'\sqrt{n}]^n$, and since 
\[\binom{d}{s/2}\geq \left(\frac{d}{s}\right)^{s/2} = \left(\frac{Cn^{1/(1-\delta)}}{1-\delta}\right)^{(1-\delta)n/2} >\left(3C'\sqrt{n}\right)^{n},\]
by the pigeonhole principle, as long as $C$ is large enough, there are two $s/2$-subsets whose corresponding combination of column vectors are the same. Subtracting the corresponding combination of column vectors leads to a nonzero $s$-sparse kernel vector of $M$ (since the indices of two $s/2$-subsets are not the same), proving the result.
\end{proof}

In our main result, we determine the (typical) asymptotic behavior up to a window of $(\log n)^{\omega(1)}$.

\begin{theorem} \label{thm:random-bound}
For any fixed $0<\delta<1$, let $n\in \mathbb{N}$ be sufficiently large, let $s= (1-\delta)n$, and let $\eps=\omega(\log\log n/\log n)$. We have that: 
\begin{enumerate}
    \item If $d\leq n^{1+1/(2-2\delta)-\eps}$ then with high probability $M_{n,d}$ is $s$-robust.
    \item If $d\geq n^{1+1/(2-2\delta)+\eps}$ then with high probability $M_{n,d}$ is not $s$-robust.
\end{enumerate}
\end{theorem}

We believe that by optimizing our bounds/similar methods, one would be able to push the bounds in \cref{thm:random-bound} up to a constant factor of $n^{1+1/(2-2\delta)}$ (though we did not focus on this aspect). It would be interesting to obtain the $1+o(1)$ multiplicative threshold behavior.

\section{Proof outline}\label{sec:outline}
We first outline the proof of \cref{thm:random-bound}.  
We will prove part (1) of \cref{thm:random-bound} over $\mathbb{F}_p$ for some prime $p=e^{\omega(\log^2 n)}$ to be chosen later (a stronger statement). Our strategy, at large, is to generate $M$ as 
\[M=\begin{pmatrix} M_1 \\ M_2\end{pmatrix}\]
where $M_1=M_{n_1,d}$ and $M_2=M_{n_2,d}$, with $n_1\approx n$ and $n_2=o(n)$. The proof consists of the following two phases:
\begin{enumerate}
    \item {\bf Phase 1:}  Given any nonzero vector $\mbf{a}\in \mathbb{F}_p^d$, we let 
    \begin{equation}\label{eq:atom-prob}\rho_{\mathbb{F}_p}(\mbf{a})=\max_{x\in \mathbb{F}_p}\P\left[\sum_{i=1}^da_i\xi_i=x\right],\end{equation}
    where the $\xi_i$s are i.i.d.~Rademacher random variables. In this phase, we will show that
    \begin{enumerate}
        \item $M_1$ is with high probability such that for all nonzero $\mbf{a}\in \mathbb{F}_p^d$, if $|\supp{\mbf a}|\leq s:=(1-\delta)n$ and $M_1\mbf{a}=\mbf{0}$, then $\rho_{\mathbb{F}_p}(\mbf{a})=e^{-\omega(\log^2 n)}$, and
        \item $M_1$ is with high probability such that every $s$-subset of its columns has rank $s-o(s)$.
    \end{enumerate}
    \item {\bf Phase 2:} Conditioned on the above properties, we will use the extra randomness of $M_2$ to show that for a specific set of $s$ columns, after exposing $M_2$, the probability that it does not have full rank is $o\left(1/\binom{d}{s}\right)$, and hence a simple union bound will give us the desired result.
\end{enumerate}
In this strategy, it turns out that {\bf Phase 1(a)} is the limiting factor, i.e., ruling out structured kernel vectors.

For the proof of the upper bound in \cref{thm:random-bound}, we exploit this observation. We show using the second-moment method that it is highly likely that some $2\lfloor(1-\delta)n/2\rfloor$ columns sum to the zero vector (corresponding to an all $1$s, highly structured kernel vector).

\section{Proof of the lower bound in \texorpdfstring{\cref{thm:random-bound}}{Theorem 1.2}}
\label{sec:lower-bound}
In this section we prove \cref{thm:random-bound}. Let (say) $p\approx e^{\log^3 n}$ be a prime, let $d=n^{1+1/(2-2\delta)-\eps}$ and $s=(1-\delta)n$ as given, and $n_1=(1-\beta)n$ where $\beta=\omega(1/\log n)$ and $\beta=o(\log\log n/\log n)$. As described in \cref{sec:outline}, our proof consists of two phases, each of which will be handled separately. 

\subsection{Phase 1: no sparse structured vectors in the kernel of \texorpdfstring{$M_1$}{M1}}\label{sec:structured-M1} 
Our first goal is to prove the following proposition. 
\begin{proposition} \label{prop:small-atom-probability}
$M_{n_1,d}$ is with high probability such that for every $(1-\delta)n$-sparse vector $\mbf{a}\in \mathbb{F}_p^d\setminus\{\mbf{0}\}$, if $M_1\mbf{a}=\mbf{0}$ then $\rho_{\mathbb{F}_p}(\mbf{a})=e^{-\omega(\log^2 n)}.$
\end{proposition}

In order to prove the above proposition, we need some auxiliary results. 

\begin{lemma} \label{lem:large-support}
$M_{n_1,d}$ is with high probability $n/\log^4n$-robust over $\mathbb{F}_p$.
\end{lemma}
\begin{proof}
Observe that for any $\boldsymbol{a}\in \mathbb{F}_p^d\setminus\{\boldsymbol{0}\}$ we trivially have that $\Pr[M_1\boldsymbol{a}=\boldsymbol{0}]\leq 2^{-n_1}=2^{-\Theta(n)}$.
  Since there are at most 
  $$\binom{d}{n/\log^4n}p^{n/\log^4n}\leq \left(\frac{edp\log^4 n}{n}\right)^{n/\log^4n}=2^{o(n)}
    $$
  $n/\log^4n$-sparse vectors $\boldsymbol{a}\in \mathbb{F}_p^d$, by a simple union bound we obtain that the the probability for such an $\boldsymbol{a}$ to satisfy $M_1\boldsymbol{a}=\boldsymbol{0}$ is $o(1)$. This completes the proof.
\end{proof}

In particular, by combining the above lemma with the Erd\H{o}s-Littlewood-Offord inequality \cite{Erd45}, we conclude that if $\boldsymbol{a}\in \mathbb{F}_p^d$ is $(1-\delta)n$-sparse and $M_1\mbf{a}=\mbf{0}$, then $\rho_{\mathbb{F}_p}(\mbf{a})=O(\log^2 n/n^{1/2})$. However, to prove Proposition \ref{prop:small-atom-probability}, we need a stronger estimate.  

The following lemma asserts that every subset of $s$ columns in $M_1$ has large rank. It will be crucial in Phase 2. 

\begin{lemma} \label{lem:large-rank}
Let $t = \omega(\log n)$. Then, with high probability $M_1=M_{n_1,d}$ is such that every subset of $s$ columns contains at least $s-t$ linearly independent columns.
\end{lemma}

\begin{proof}
Consider the event that one such subset has rank at most $s-t$. There are $\binom{d}{s}\le d^s\le n^n$ possible choices of columns. For each such choice, there are at most $2^s\le 2^n$ ways to choose a spanning set of $r\le s-t$ columns. Such a subset has span containing at most $2^s$ many $\{\pm1\}$ vectors (indeed, consider a full-rank $r\times r$ sub-block; any $\{\pm1\}$ vector in the span of the columns is determined by its value on these $r$ coordinates), so the probability that the remaining at least $t = \omega(\log n)$ columns are in the span is at most $(2^s/2^{n_1})^t\le(2^{-(\delta-\beta)n})^t = o(n^{-n})$. Taking a union bound, the result follows.
\end{proof}

Next, we state a version of Hal\'asz's inequality (\cite[Theorem~3]{Hal77}) as well as a ``counting inverse Littlewood-Offord theorem'' as was developed in \cite{FJLS21}.
\begin{definition}
Let $\bs{a}\in\mb{F}_p^n$ and $k\in\mb{N}$. We define $R_k^{\ast}(\bs{a})$ to be the number of solutions to 
\[\pm a_{i_1}\pm a_2\pm \ldots\pm a_{i_{2k}}\equiv 0 \mod p\]
with $|\{i_1,\ldots,i_{2k}\}|>1.01k$.
\end{definition}
\begin{theorem}[{\cite[Theorem~1.4]{FJLS21}}]\label{thm:halasz-fp}
Given an odd prime $p$, integer $n$, and vector $\bs{a}=(a_1,\ldots, a_n) \in\mb{F}_p^{n}\setminus\{\bs{0}\}$, suppose that an integer $0\le k \le n/2$ and positive real $L$ satisfy $30L \le |\supp{(\bs{a})}|$ and $80kL \le n$. Then
\[\rho_{\mb{F}_p}(\bs{a})\le\frac{1}{p}+C_{\ref{thm:halasz-fp}}\frac{R_k^\ast(\bs{a}) + ((40k)^{0.99}n^{1.01})^k}{2^{2k} n^{2k} L^{1/2}} + e^{-L}. \]
\end{theorem}

We denote $\mbf{b} \subset \mbf{a}$  if $\mbf{b}$ is a subvector of $\mbf{a}$ and let $| \mbf{b}|$ be the size of the support of a vector $\mbf{b}$.
\begin{theorem}[{\cite[Theorem~1.7]{FJLS21}}]
  \label{thm:counting}
  Let $p$ be a prime, let $k, n \in \mathbb{N}$, $s\in [n]$ and $t\in [p]$. Define $\Bad_{k,m,\geq t}(s,d)$ as the following set:
  \begin{align*}
     \left\{\boldsymbol{a} \in \mathbb{F}_{p}^{d} : |\mbf{a}|\leq s, \textrm{ and }R_k^{\ast}(\boldsymbol{b})\geq t\cdot \frac{2^{2k} \cdot |\boldsymbol{b}|^{2k}}{p} \text{ for every } \boldsymbol{b}\subseteq \boldsymbol{a} \text{ with } |\boldsymbol{b}|\geq m\right\},
  \end{align*}
 
  We have
  \[
    |\Bad_{k,m,\geq t}(s,d)| \leq \binom{d}{s}\left(\frac{m}{s}\right)^{2k-1} (1.01t)^{m-s}p^{s}.
  \]
\end{theorem}

We now are in position to prove \cref{prop:small-atom-probability}. The proof is quite similar to the proofs in \cite{FJ19,FJLS21,FLM21}. 

\begin{proof}[Proof of \cref{prop:small-atom-probability}]
Let $k = \log^3 n$ and  $m = n/\log^4 n, p\approx e^{\log^3 n}$.

First we use \cref{lem:large-support} to rule out vectors $\bs{a}$ with a support of size less than $n/\log^4n$. Next, let (say) $L = n/\log^{10}n$ and let $\sqrt{L}\leq t\leq p$.

Consider a fixed $\bs{a}\in \Bad_{k,m,\geq t}(s,d)\setminus \Bad_{k,m,\geq 2t}(s,d)$ and we wish to bound the probability that $M_1\mbf{a}=\mbf{0}$. By definition, there is a set $S\subseteq\on{supp}(\bs{a})$ of size at least $m$ such that
\begin{equation}\label{eq:Rk-bound}
R_k^\ast(\bs{a}|_S) < 2t\cdot\frac{2^{2k}|S|^{2k}}{p}.
\end{equation}

Since the rows are independent and since $\rho_{\mathbb{F}_p}(\mbf{a})\leq \rho_{\mathbb{F}_p}(\mbf{a}|_S)$, the probability that $M_1\mbf{a}=\mbf{0}$ is at most $\rho_{\mb{F}_p}(\bs{a}|_S)^{n_1}$. Furthermore, by \cref{thm:halasz-fp} and the given conditions, which guarantee $30L\le m\le|\on{supp}(\bs{a}|_S)|$ and $80kL\le m\le|S|$, and by $\sqrt{L}\le t\le p$, we have
\begin{align}
\rho_{\mb{F}_p}(\bs{a}|_S)&\le\frac{1}{p}+C_{\ref{thm:halasz-fp}}\frac{R_k^\ast(\bs{a}|_S) + ((40k)^{0.99}|S|^{1.01})^k}{2^{2k} |S|^{2k} L^{1/2}} + e^{-L}\notag\\
&\le\frac{1}{p} + \frac{2C_{\ref{thm:halasz-fp}}t}{p\sqrt{L}} + \frac{10^k C_{\ref{thm:halasz-fp}}}{L^{1/2}}\bigg(\frac{k}{|S|}\bigg)^{0.99k} + e^{-L}\notag\\
&\le\frac{Ct}{p\sqrt{L}}\label{eq:rho-from-Rk-bound}
\end{align}
for all sufficiently large $n$ by \cref{eq:Rk-bound}.
All in all, taking a union bound over all the possible choices of $\bs{a}$ (\cref{thm:counting}), and using the fact that $s= (1-\delta)n$ and $n_1=(1-\beta)n$ with $\beta=\omega(1/\log n)$, we obtain the bound
\begin{align*}
\binom{d}{s}\left(\frac{m}{s}\right)^{2k-1} (1.01t)^{m-s}p^{s}&\bigg(\frac{Ct}{p\sqrt{L}}\bigg)^{n_1}
\leq \left(\frac{ed}{s}\right)^{s}(1.01t)^{m}\left(\frac{p}{1.01t}\right)^{s}\bigg(\frac{Ct}{p\sqrt{L}}\bigg)^{(1-\beta)n}\\
&\leq \left(\frac{ed}{(1-\delta)n}\right)^{(1-\delta)n}2^{o(n)}\bigg(\frac{1.01 t}{p}\bigg)^{(\delta-\beta) n}\left(\frac{C(\log n)^5}{\sqrt{n}}\right)^{(1-\beta)n}\\
&= o(1/p)
\end{align*}
on the probability $M_1$ has such a kernel vector for sufficiently large $n$. Here we used the bounds $d\leq n^{1+1/(2-2\delta)-\eps}$, $\eps=\omega (\log\log n/\log n)$  and $\beta=o(\eps)$. Union bounding over all possible values of $t$ shows that there is an appropriately small chance of having such a vector for any $t\ge\sqrt{L}$.

Finally, note that $B_{k,m,\ge p}(s,d)$ is empty and thus the above shows that kernel vectors $\mbf{a}$ cannot be in $\mbf{B}_{k,m,\ge\sqrt{L}}(s,d)$. A similar argument as in \cref{eq:Rk-bound} and \cref{eq:rho-from-Rk-bound} shows that
\[\rho_{\mb{F}_p}(\mbf{a})\le\frac{C'}{p},\]
and the result follows.
\end{proof}

\subsection{Phase 2: boosting the rank using \texorpdfstring{$M_2$}{M2}}






Here we show that, conditioned on the the conclusions of \cref{prop:small-atom-probability} and \cref{lem:large-rank}, after exposing $M_2$ with high probability $M=\begin{pmatrix}
M_1\\ M_2 \end{pmatrix}$ is $s$-robust.

To analyze the probability that a given subset of $s$ columns is not of full rank, we will use the following procedure: 

Fix any subset of $s$ columns in $M_1$, and let $C:=(\boldsymbol{c}_1,\dots,\mbf{c}_{s})$ be the submatrix in $M_1$ that consists of those columns. We reveal $M_2$ according to the following steps:
\begin{enumerate}
    \item Let $I\subseteq [s]$ be the largest subset of indices such that the columns $\{\mbf c_i \mid i\in I\}$ are linearly independent. By \cref{lem:large-rank} we have that  $T:=|I|\geq s-t= (1-\delta)n-t$, where $t=\omega(\log n)$. Without loss of generality we may assume that $I:=\{\mbf c_1,\ldots,\mbf c_T\}$ and $T\leq s-1$ (otherwise we have already found $s$ independent columns of $M$). By maximality, we know that $\mbf c_{T+1}$ can be written (uniquely) as a linear combination of $\mbf c_1,\dots, \mbf c_{T}$. That is, there exists a unique combination for which $\sum^{T}_{i=1} x_i \mbf c_i= \mbf c_{T+1}$. In particular, this means that 
    \begin{align*}
    \sum_{i=1}^Tx_i \mbf c_i-\mbf c_{T+1}=0,    
    \end{align*}
    and hence the vector $\boldsymbol{x}=(0,\dots, x_1,\dots,x_{T}, -1, \dots,0)^T\in \mathbb F_{q}^d$ is $(T+1)$-sparse and satisfies $M_1\boldsymbol{x}=0$.
    Since $T+1\leq s$, by \cref{prop:small-atom-probability} we know that $\rho_{\mathbb{F}_p}(\boldsymbol{x})=2^{-\omega(\log^2 n)}$. 
    \item Expose the row vector of dimension $T+1$ from $M_2$ below the matrix $(\mbf c_1,\dots, \mbf c_{T+1})$.  We obtain a matrix of size $(n_1+1)\times (T+1)$. Denote the new row as $(y_1,\dots,y_{T+1})$. 
    \item If the new matrix is of rank $T+1$, then consider this step as a ``success'', expose the entire row and start over from $(1)$. Otherwise, consider this step as a ``failure'' (As we failed to increase the rank) and observe that if $\begin{bmatrix}
    \mbf c_1 &\dots & \mbf c_{T+1}\\
     y_1 &\dots &y_{T+1}
    \end{bmatrix}$ is not of full rank, then we must have \begin{equation}
       x_1y_1+x_2y_2+\ldots-y_{T+1}=0. \nonumber
    \end{equation}
The probability to expose such a vector $y$ is at most $\rho_{\mathbb{F}_p}(\mbf{x})=e^{-\omega(\log^2 n)}$.
    \item All in all, the probability for more than $\beta n-t$ failures is at most $\binom{\beta n}{t}\left(e^{-\omega(\log^2 n)}\right)^{\beta n-t}=e^{-\omega(n\log n)}=o\left(\binom{d}{s}^{-1}\right)$. Therefore, by the union bound we obtain that with high probability $M$ is $s$-robust.  
\end{enumerate}

This completes the proof.

\section{Proof of the upper bound in \texorpdfstring{\cref{thm:random-bound}}{Theorem 1.2}}
We first perform preliminary computations to compute a certain correlation. This boils down to estimating binomial sums. Let $\xi_i,\xi_i'$ be independent Rademacher variables and define
\[\alpha(n,m) = \frac{\mb{P}[\xi_1+\cdots+\xi_n=\xi_1+\cdots+\xi_m+\xi_{m+1}'+\cdots+\xi_n'=0]}{\mb{P}[\xi_1+\cdots+\xi_n=0]^2}.\]
Clearly $\alpha(n,m)\le\alpha(n,n)\le 10\sqrt{n}$ by \cite{Erd45}.
\begin{lemma}\label{lem:rad-sum-correlation-1}
Fix $\lambda > 0$. If $n$ is even and $0\le m\le (1-\eps)n$ we have
\[\alpha(n,m) = 1+O(m/(\eps n)).\]
\end{lemma}
\begin{proof}
We have
\[\alpha(n,m)\le\frac{\sup_k\mb{P}[\xi_1+\cdots+\xi_{n-m}=k]}{\mb{P}[\xi_1+\cdots+\xi_n=0]^2}\le\frac{2^{-(n-m)}\binom{n-m}{\lfloor(n-m)/2\rfloor}}{2^{-n}\binom{n}{n/2}}= 1 + O(m/(n-m)).\qedhere\]
\end{proof}
We will also need a more refined bound when $m$ is small.
\begin{lemma}\label{lem:rad-sum-correlation-2}
If $n$ is even and $0\le m\le n^{1/2}$, we have
\[\alpha(n,m) = 1+O(m^2/n^2).\]
\end{lemma}
\begin{proof}
Using the approximation $1-x=\exp(-x-x^2/2+O(x^3))$ for $|x|\le 1/2$ we see that if $y$ is an integer satisfying $1\le y\le x/2$ then
\begin{align}
x(x-1)\cdots(x-y+1) &= x^y\exp\bigg(-\sum_{i=0}^{y-1}\frac{i}{x}-\sum_{i=0}^{y-1}\frac{i^2}{2x^2}+O\Big(\frac{y^4}{x^3}\Big)\bigg)\notag\\
&= x^y\exp\bigg(-\frac{y(y-1)}{2x}-\frac{y(y-1)(2y-1)}{12x^2} + O\Big(\frac{y^4}{x^3}\Big)\bigg).\label{eq:stirling-gen}
\end{align}
We now apply this to the situation at hand. We see $\alpha(n,m)$ is equal to
\begin{align*}
& \frac{2^{-(2n-m)}\sum_{k=0}^m\binom{m}{k}\binom{n-m}{n/2-k}^2}{2^{-2n}\binom{n}{n/2}^2}\\
=& 2^m\sum_{k=0}^m\binom{m}{k}\bigg(\frac{(n/2)(n/2-1)\cdots(n/2-k+1)\times(n/2)(n/2-1)\cdots(n/2-(m-k)+1)}{n(n-1)\cdots(n-m+1)}\bigg)^2\\
=& 2^m\sum_{k=0}^m\binom{m}{k}\bigg(\frac{(n/2)^me^{-\frac{k(k-1)}{n}-\frac{k(k-1)(2k-1)}{3n^2}-\frac{(m-k)(m-k-1)}{n}-\frac{(m-k)(m-k-1)(2m-2k-1)}{3n^2}+O(m^4/n^3)}}{n^me^{-\frac{m(m-1)}{2n}-\frac{m(m-1)(2m-1)}{12n^2}+O(m^4/n^3)}}\bigg)^2\\
=&2^{-m}\sum_{k=0}^m\binom{m}{k}\exp\bigg(-\frac{m^3-4mk(m-k)+n(2k-m)^2-nm}{2n^2}+O(m^2/n^2)\bigg)\\
=& 2^{-m}\sum_{k=0}^m\binom{m}{k}\bigg(1-\frac{m^3-4mk(m-k)-nm}{2n^2}+O(m^2/n^2)\bigg)\bigg(1-\frac{(2k-m)^2}{2n}+O\bigg(\frac{(2k-m)^4}{n^2}\bigg)\bigg)\\
=& 2^{-m}\sum_{k=0}^m\binom{m}{k}\bigg(1-\frac{m^3-4mk(m-k)-nm}{2n^2}\bigg)\bigg(1-\frac{(2k-m)^2}{2n}\bigg)+O(m^2/n^2).
\end{align*}
In the third line, we used \cref{eq:stirling-gen} and in the fourth line, we simplified the expression and used $k\le m\le n^{1/2}$ to subsume many terms into an error of size $O(m^2/n^2)$. The fifth line used $\exp(x)=1+x+O(x^2)$ for $|x|\le 1$ and the sixth line uses $2^{-m}\binom{m}{k}(2k-m)^4\le 2m^2\exp(-(2k-m)^2/100)$. Finally, this sum equals
\[\alpha(n,m) = 1-\frac{3nm^2-3m^3+2m^2}{4n^3}+O(m^2/n^2) = 1+O(m^2/n^2).\qedhere\]
\end{proof}

We are ready to prove the upper bound in \cref{thm:random-bound}.
\begin{proof}[Proof of the upper bound in \cref{thm:random-bound}]
We are given $\delta\in(0,1)$ and $\eps = \omega(\log\log n/\log n)$, with $d = n^{1+1/(2-2\delta)+\eps}$. Let $s = 2\lfloor(1-\delta)n/2\rfloor$. We consider an $n\times d$ random matrix with independent Rademacher entries and wish to show it is not $s$-robust with high probability. We may assume $\eps < 1/2$ as increasing $d$ makes the desired statement strictly easier.

For an $s$-tuple of columns labeled by the index set $S\subseteq[d]$, let $X_S$ be the indicator of the event that these columns sum to the zero vector. Let $X = \sum_{S\in\binom{[d]}{s}}X_S$, and let $(\xi_1,\ldots,\xi_d)$ be a vector of independent Rademachers. We have
\[\mb{E}X = \binom{d}{s}\mb{E}X_{[s]} = \binom{d}{s}\mb{P}[\xi_1+\cdots+\xi_s=0]^n = \binom{d}{s}\bigg(2^{-s}\binom{s}{s/2}\bigg)^n\]
and
\begin{align*}
\on{Var}X &= \mb{E}X^2-(\mb{E}X)^2 = \sum_{S,T\in\binom{[d]}{s}}\Big(\mb{P}\Big[\sum_{i\in S}\xi_i=\sum_{j\in T}\xi_T=0\Big]^n-\mb{P}[\xi_1+\cdots+\xi_s=0]^{2n}\Big)\\
&= (\mb{E}X)^2\cdot\frac{1}{\binom{d}{s}^2}\sum_{S,T\in\binom{[d]}{s}}\Bigg(\frac{\mb{P}\Big[\sum_{i\in S}\xi_i=\sum_{j\in T}\xi_T=0\Big]^n}{\mb{P}[\xi_1+\cdots+\xi_n=0]^{2n}}-1\Bigg)\\
&= (\mb{E}X)^2\sum_{m=0}^s\frac{\binom{s}{m}\binom{d-s}{s-m}}{\binom{d}{s}}\cdot(\alpha(s,m)^n-1).
\end{align*}
For every $\eta>0$ and $m\le c_\eta n^{1/2}$, where $c_\eta$ is a sufficiently small absolute constant in terms of $\eta$, we see $|\alpha(s,m)^n-1|\le\eta$ by \cref{lem:rad-sum-correlation-2}. For $c_\eta n^{1/2} < m\le(1-\eps/8)s$ we have $\alpha(s,m)^n\le\exp(O(m/\eps))$ by \cref{lem:rad-sum-correlation-1}. For this range we have, since $m/s\ge n^{\delta/2}s/d$,
\[\frac{\binom{s}{m}\binom{d-s}{s-m}}{\binom{d}{s}}\le(s+1)\mb{P}[\mr{Bin}(s,s/d)\ge m]\le\exp(-s D(m/(2s)||s/d))\le\exp(-m(\delta/4)\log n)\]
by Chernoff--Hoeffding (the fact that $\mr{Bin}(n,p)$ exceeds $nq$ for $q\ge p$ with probability at most $\exp(-nD(q||p))$, where this is the KL-divergence). Thus
\[\sum_{m=c\sqrt{n}}^{(1-\eps)s}\frac{\binom{s}{m}\binom{d-s}{s-m}}{\binom{d}{s}}\cdot(\alpha(s,m)^n-1)\le\sum_{m=c\sqrt{n}}^{(1-\eps)s}\exp(O(m/\eps))\cdot\exp(-m(\delta/4)\log n) = o(1)\]
as $\eps = \omega(\log\log n/\log n)$.

Finally for $(1-\eps/8)s\le m\le s$ we have
\[\sum_{m=(1-\eps)s}^s\frac{\binom{s}{m}\binom{d-s}{s-m}}{\binom{d}{s}}\cdot(\alpha(s,m)^n-1)\le\sum_{m=(1-\eps)s}^s\frac{\binom{s}{m}\binom{d-s}{s-m}}{\binom{d}{s}}(10\sqrt{n})^n\le 2^{s}\frac{\binom{d}{\eps s/8}}{\binom{d}{s}}(10\sqrt{n})^n.\]
Thus
\[\sum_{m=(1-\eps)s}^s\frac{\binom{s}{m}\binom{d-s}{s-m}}{\binom{d}{s}}\cdot(\alpha(s,m)^n-1)\le\bigg(\frac{10s}{\eps d}\bigg)^{(1-\eps/8)s}(10\sqrt{n})^n\le(n^{-\frac{1}{2-2\delta}-\eps/2})^{(1-\eps/8)(1-\delta)n}(10\sqrt{n})^n,\]
since $d = n^{1+1/(2-2\delta)+\eps}$ and $s = 2\lfloor(1-\delta)n/2\rfloor$ along with $\eps = \omega(\log\log n/\log n)$. We see that this is $o(1)$. Thus
\[\on{Var}X\le(\mb{E}X)^2\cdot\bigg(\eta + o(1) + o(1)\bigg)\le 2\eta (\mb{E}X)^2\]
for $n$ sufficiently large, and thus $X>0$ with probability at least $1-2\eta$.
\end{proof}

\bibliographystyle{abbrv}
\bibliography{main}

\begin{thebibliography}{10}

\bibitem{zhang2019analysis}
A.~Abdi, F.~Fekri, and H.~Zhang.
\newblock Analysis of sparse-integer measurement matrices in compressive
  sensing.
\newblock In {\em ICASSP 2019-2019 IEEE International Conference on Acoustics,
  Speech and Signal Processing (ICASSP)}, pages 4923--4927. IEEE, 2019.

\bibitem{BTV10}
J.~Bourgain, V.~H. Vu, and P.~M. Wood.
\newblock On the singularity probability of discrete random matrices.
\newblock {\em J. Funct. Anal.}, 258(2):559--603, 2010.

\bibitem{CT06}
E.~J. Candes and T.~Tao.
\newblock Near-optimal signal recovery from random projections: universal
  encoding strategies?
\newblock {\em IEEE Trans. Inform. Theory}, 52(12):5406--5425, 2006.

\bibitem{RHE14}
Y.~C. Eldar, A.~M. Haimovich, and M.~Rossi.
\newblock Spatial compressive sensing for {MIMO} radar.
\newblock {\em IEEE Trans. Signal Process.}, 62(2):419--430, 2014.

\bibitem{Erd45}
P.~Erd\"{o}s.
\newblock On a lemma of {L}ittlewood and {O}fford.
\newblock {\em Bull. Amer. Math. Soc.}, 51:898--902, 1945.

\bibitem{FJ19}
A.~Ferber and V.~Jain.
\newblock Singularity of random symmetric matrices---a combinatorial approach
  to improved bounds.
\newblock {\em Forum Math. Sigma}, 7:Paper No. e22, 29, 2019.

\bibitem{FJLS21}
A.~Ferber, V.~Jain, K.~Luh, and W.~Samotij.
\newblock On the counting problem in inverse {L}ittlewood-{O}fford theory.
\newblock {\em J. Lond. Math. Soc. (2)}, 103(4):1333--1362, 2021.

\bibitem{FLM21}
A.~Ferber, K.~Luh, and G.~McKinley.
\newblock Resilience of the rank of random matrices.
\newblock {\em Combin. Probab. Comput.}, 30(2):163--174, 2021.

\bibitem{fukshansky2019covering}
L.~Fukshansky and A.~Hsu.
\newblock Covering point-sets with parallel hyperplanes and sparse signal
  recovery.
\newblock {\em Discrete \& Computational Geometry}, 2022.

\bibitem{FNS19}
L.~Fukshansky, D.~Needell, and B.~Sudakov.
\newblock An algebraic perspective on integer sparse recovery.
\newblock {\em Appl. Math. Comput.}, 340:31--42, 2019.

\bibitem{Hal77}
G.~Hal\'{a}sz.
\newblock Estimates for the concentration function of combinatorial number
  theory and probability.
\newblock {\em Period. Math. Hungar.}, 8(3-4):197--211, 1977.

\bibitem{he2010simplest}
M.~Haseyama, Z.~He, and T.~Ogawa.
\newblock The simplest measurement matrix for compressed sensing of natural
  images.
\newblock In {\em 2010 IEEE International Conference on Image Processing},
  pages 4301--4304. IEEE, 2010.

\bibitem{Iwe14}
M.~A. Iwen.
\newblock Compressed sensing with sparse binary matrices: instance optimal
  error guarantees in near-optimal time.
\newblock {\em J. Complexity}, 30(1):1--15, 2014.

\bibitem{XLJZ15}
Y.~Jiang, X.-J. Liu, S.-T. Xia, and H.-T. Zheng.
\newblock Deterministic constructions of binary measurement matrices from
  finite geometry.
\newblock {\em IEEE Trans. Signal Process.}, 63(4):1017--1029, 2015.

\bibitem{KS20}
S.~Konyagin and B.~Sudakov.
\newblock An extremal problem for integer sparse recovery.
\newblock {\em Linear Algebra Appl.}, 586:1--6, 2020.

\bibitem{Spe85}
J.~Spencer.
\newblock Six standard deviations suffice.
\newblock {\em Trans. Amer. Math. Soc.}, 289(2):679--706, 1985.

\bibitem{Tik18}
K.~Tikhomirov.
\newblock Singularity of random {B}ernoulli matrices.
\newblock {\em Ann. of Math. (2)}, 191(2):593--634, 2020.

\end{thebibliography}

\end{document}